\documentclass[12pt]{article}
\usepackage{amsmath}
\usepackage{amssymb}
\usepackage{theorem}
\usepackage{graphicx}

\numberwithin{equation}{section}

\newtheorem{thm}{Theorem}[section]
\newtheorem{lemma}[thm]{Lemma}
\newtheorem{prop}[thm]{Proposition}
\newtheorem{cor}[thm]{Corollary}
{\theorembodyfont{\rmfamily}
\newtheorem{defn}[thm]{Definition}

\newtheorem{rmk}[thm]{Remark}
}

\newcommand{\qed}{\hfill \mbox{\raggedright \rule{.07in}{.1in}}}
 
\newenvironment{proof}{\vspace{1ex}\noindent{\bf
Proof}\hspace{0.5em}}{\hfill\qed\vspace{1ex}}
\newenvironment{pfof}[1]{\vspace{1ex}\noindent{\bf Proof of
#1}\hspace{0.5em}}{\hfill\qed\vspace{1ex}}

\newcommand{\R}{{\mathbb R}}
\newcommand{\C}{{\mathbb C}}
\newcommand{\Z}{{\mathbb Z}}

\newcommand{\var}{\operatorname{var}}
\newcommand{\BV}{\operatorname{BV}}

\newcommand{\SMALL}{\textstyle}

\newcommand{\cC}{{\mathcal C}}

\title{Statistical Properties and Decay of Correlations for Interval Maps with Critical Points and 
Singularities.}

\author{
Stefano Luzzatto
\thanks{Mathematics Section,
Abdus Salam International Centre for Theoretical Physics,
Strada Costiera 11, 34151 Trieste, Italy. \texttt{luzzatto@ictp.it}}
\and 
Ian Melbourne 
\thanks{Department of Mathematics, University of Surrey,
Guildford, Surrey GU2 7XH, UK. Current Address: Mathematics Institute, University of Warwick, Coventry CV4 7AL, UK. \texttt{i.melbourne@warwick.ac.uk}}
}

\date{10 October 2011; revised 11 October 2012}

\begin{document}

\maketitle

 \begin{abstract}
We consider a class of piecewise smooth one-dimensional maps with critical
points and singularities (possibly with infinite derivative).
Under mild summability conditions on the growth of the derivative on critical orbits, we prove the central limit theorem and a vector-valued
almost sure invariance principle.  We also obtain results on decay of correlations and large deviations.
  \end{abstract}

\section{Introduction and statement of results}

It is well known that deterministic dynamical systems can exhibit ``chaotic'' and ``random-like'' behaviour which in many cases gives rise to  statistical properties analogous to those which are classical in the probabilistic setting of sequences of random variables. 
The purpose of this paper is to establish some statistical properties  for a 
large class of deterministic interval maps (studied by
Ara\'ujo, Luzzatto and Viana~\cite{ALV09})
which are allowed to have critical points of arbitrary order, discontinuities and even unbounded derivative.

The method in~\cite{ALV09} yields the existence of
an absolutely continuous invariant probability measure 
via the standard approach of {\em inducing}. However
the induced map does not
have many desirable properties such as uniformly bounded distortion, long
branches, or a Markov partition.    This accounts for the generality of the method, but has the
drawback (at least temporarily) that standard statistical properties such as the central limit
theorem (CLT)  are not an immediate consequence.  
In contrast, D\'iaz-Ordaz
{\em et al.}~\cite{DiazHollandLuzzatto06} establish the existence of a
Gibbs-Markov induced map (Markov map with uniformly bounded distortion and long
branches) for a more restricted class of maps.
The CLT is then guaranteed provided that
the associated return time function lies in $L^2$ (in~\cite{DiazHollandLuzzatto06} the return time function has exponential tails and hence lies in $L^p$ for all $p<\infty$).
The vector-valued almost sure invariance
principle (ASIP) is also immediate in their setting by~\cite{Gouezel10,MN09}.

The aim of this paper is to establish statistical properties 
in the more general setting of~\cite{ALV09}.    Under very mild 
strengthenings of the hypotheses in~\cite{ALV09} (primarily to guarantee sufficient
integrability for the return time function), we prove the CLT and vector-valued
ASIP.   In addition, we obtain results on decay of correlations and large deviations.

In  \S \ref{subsec:defs} we give  the definition of the class of maps under consideration.  
In \S\S \ref{subsec:limit} and~\ref{subsec:decay} we 
state precisely our main results on statistical limit laws
and decay of correlations respectively.
In \S \ref{subsec:strategy} we give an overview of the strategy for the proofs.

\subsection{Interval maps with critical points and singularities}
\label{subsec:defs}

Let $I=[0,1]$ and suppose $f:I\to I$ is
a piecewise smooth map in the sense that $I$ is a union of finitely many intervals such that $f$ is $C^2$ and monotone on the interior of each interval and extends continuously to the boundary of each interval.
\begin{figure}[ht]
\center{
\includegraphics{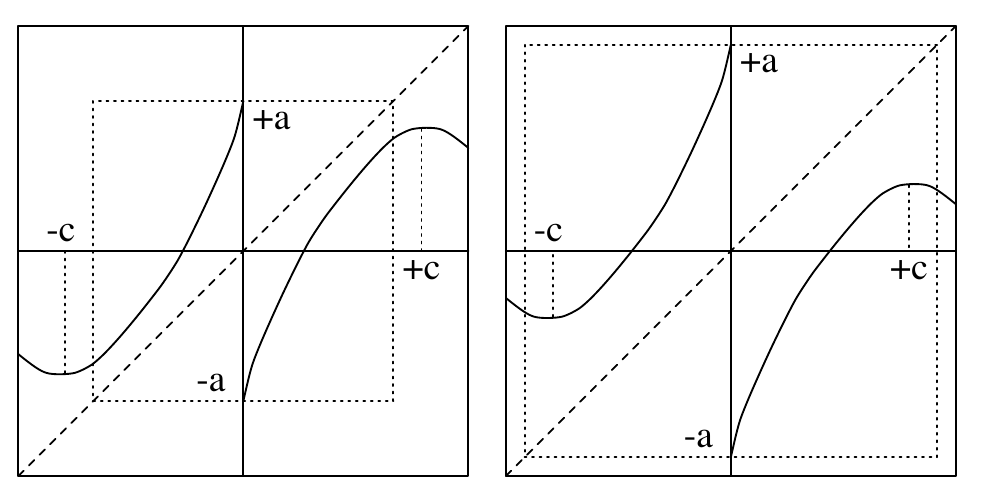}
}
\caption{Interval maps with critical points and singularities}
\end{figure}
To define our specific assumptions on $f$, we introduce some notation. 
Let  \( \mathcal C\)  denote the set of all ``one-sided'' critical/singular points 
 $c^+$ and $c^{-}$ ($\mathcal{C}$ is assumed to be finite)
 and define corresponding one-sided neighbourhoods
\[
\Delta(c^{+},\delta) = (c^{+},c^{+}+\delta) \quad\text{and}\quad
\Delta(c^{-},\delta) = (c^{-}-\delta,c^{-}),
\]
for each $\delta>0$. For simplicity, from now on we use $c$ to
represent 
the generic element of \( \mathcal C\) and write $\Delta$ for
$\cup_{c\in\cC} \Delta(c,\delta)$.  

We can now state our assumptions. 

\begin{description}
\item[(A1) Nondegenerate critical/singular set.]
\emph{Each 
\( c\in\mathcal C\) has a well-defined (one-sided) critical order
$\ell = \ell(c) > 0$  in the sense that }
\begin{equation}\label{eq:der0}
 |f(x)-f(c)|
\approx  d(x,c)^{\ell}, \quad
 |f'(x)|\approx  d(x,c)^{\ell-1},  \quad 
 |f''(x)|\approx  d(x,c)^{\ell-2}
\end{equation}
\emph{for all $x$ in some $\Delta(c,\delta)$.  We suppose \( \ell(c)\neq 1 \) for all \( c\in\mathcal C \).}
\end{description}

Note that we say that \(
f\approx g \) if  the ratio \( f/g \) is bounded above and below
uniformly in the stated domain. 
The case $\ell(c)<1$ corresponds to \emph{singular} points and the case
$\ell(c)>1$ to \emph{critical} points.
The case \( \ell(c) = 1 \) is excluded since
this would be a degenerate case which is not hard to deal with
but would require having to introduce special notation and special
arguments,  whereas the other cases can all be dealt with in
a unified formalism.

\begin{description}
\item[(A2) Uniform expansion 
away from the critical/singular set.]
\emph{There exists a constant \( \kappa>0 \), independent of \(
\delta \), such that for every point \( x \) and every integer \(
n\geq 1 \) such that $d(f^jx,\mathcal{C})>\delta$ for all $0\le j \le n-1$ and $d(f^nx,\mathcal{C}) \leq \delta$ we have }
\begin{equation*}
    |(f^n)'(x)| \geq \kappa.
\end{equation*}
\emph{Further,  for every $\delta>0$ there exist constants $c(\delta)>0$
and $\lambda(\delta)>0$ such that }
\begin{equation*}
|(f^n)'(x)| \ge c(\delta) e^{\lambda(\delta) n}
\end{equation*}
\emph{for every $x$ and $n\ge 1$ such that $d(f^jx,\mathcal{C})>\delta$ for all $0\le j \le n-1$.}
\end{description}

These conditions are quite natural and are generally satisfied for smooth maps. 
For example, the first  one is satisfied 
if \( f \) is \( C^3 \), has negative Schwarzian derivative and 
satisfies the property that the the derivative along all critical
orbits tends to infinity~\cite[Theorem~1.3]{BruStr03}. 
The second  is satisfied if
$f$ is  $C^2$ and all periodic points are
repelling~\cite{Man85}. 

The next two conditions generalise condition {\bf (A3) Summability condition
along critical orbits} from~\cite{ALV09}.
Write 
\[
D_n(c):=|(f^n)'(fc)|, \qquad 
E_n(c):=D_{n-1}(c)^{1/(2\ell(c)-1)}, \qquad
d_n(c):=d(f^nc,\mathcal{C}).
\]
These quantities are abbreviated to $D_n$, $E_n$, $d_n$ in future.
Fix $p\ge0$.
\emph{For every critical point \( c \) with \( \ell = \ell(c) > 1 \) 
we have }
\begin{description}
\item[(A3\( _p \))]
$\sum_{n=1}^\infty n^pd_n^{-1}\log d_n^{-1}\,E_n^{-1} < \infty$.
\item[(A4\( _p \))]
$\sum_{n=1}^\infty n^pE_n^{-1} < \infty$.

\end{description}
We do not impose conditions on the orbit of the singular points with \( \ell(c) < 1 \).

Notice that (A3\(_p \)) implies (A4$_{p'}$) if \( p\geq p' \) but not if \( p< p' \). The two conditions play different roles in the arguments and we will sometimes assume they hold for distinct values \( p \) and \( p' \).   
Condition (A3\(_p \))  plays off the derivative against the recurrence in
such a way as to  optimize to some extent the class of maps to which
it applies. As mentioned in \cite{ALV09}  we cannot
expect to obtain certain dynamical properties in this setting
using a condition which only takes into account the growth of the
derivative. 
Note that the conditions are satisfied for any \( p\geq 0 \) 
 if the derivative is growing exponentially fast  and the recurrence
is not faster than 
exponential in the sense that there is a constant $C>0$  such that
\(
D_n\ge C e^{\lambda n}\) and  \(  d_n
\ge C e^{-\alpha n}\) 
 with \(  \alpha < {\lambda}/{(2\ell-1)}.
 \)

\begin{thm}[Ara\'ujo {\em et al.}~\cite{ALV09}] \label{thm-ALV}
Suppose that \( f: I \to I  \) satisfies assumptions  (A1), (A2), (A3$_0$) 
and (A4\( _1 \)). Then it admits an ergodic, absolutely continuous, invariant probability measure \( \mu \).
\end{thm}

The proof of Theorem~\ref{thm-ALV} relies on the construction of a suitable
partition \( \alpha \) of the interval \( I \) (mod 0)  into a finite or  countable collection of subintervals and a return time function \( \tau: I\to \mathbb N \) which is constant on elements of the partition \( \alpha \). 
Define the \emph{induced map} \( F: I \to I \) 
given by \( F|_{\alpha}=f^{\tau(\alpha)} \). 
Ara\'ujo {\em et al.} construct the induced map so that (i) $F:I\to I$ is 
a Rychlik map~\cite{Rychlik83} and (ii) $\tau$ is Lebesgue integrable.
It follows from Rychlik~\cite{Rychlik83} that there is an absolutely
continuous invariant probability measure for the induced map $F$. 
Then Theorem~\ref{thm-ALV} follows by standard arguments
from the integrability of the return time $\tau$.

We note that (A3\( _{0} \)) and (A4$_0$) are used to show that $F$ is Rychlik, and (A4$_1$) implies that $\tau$ is integrable.
Moreover, (A4\( _p \)) implies
that $\tau\in L^p$ (see Corollary~\ref{cor-ALV}(a) below).
It is natural to explore statistical limit laws such as the central limit theorem when $\tau\in L^2$.

\subsection{Statement of results: Statistical limit laws}
\label{subsec:limit}

The first main result of this paper is a vector-valued almost sure invariance principle. 
For $d\ge1$, we let
$\phi:I\to\R^d$ be an observable lying in the space $\BV$ of functions of bounded
variation.     Suppose that $\int_I\phi\,d\mu=0$, and define
$\phi_n=\sum_{i=0}^{n-1}\phi\circ f^i$.

\begin{defn}
The sequence \( \{\phi_n\} \) 
satisfies the \emph{$d$-dimensional almost sure invariance principle (ASIP)} if there exists a \( \lambda<\frac12 \) and an abstract probability space supporting a sequence of random variables \( \{\phi^{*}_n\} \) and a \( d \)-dimensional Brownian motion \( W(t) \) such that: 
\begin{enumerate}
\item \( \{\phi_n \}_{n\geq 1} =_d \{\phi^{*}_n\}_{n\geq 1}\);
\item \( \phi^{*}_n = W(n) + \mathcal O(n^{\lambda})\)  as \( n\to \infty \) a.e.
\end{enumerate}
(The ASIP is \emph{nondegenerate} if the Brownian motion \( W(t) \) has a nonsingular covariance matrix.)
\end{defn}

\begin{thm}\label{thm-ASIP}
Suppose that \( f: I \to I  \) satisfies assumptions  (A1), (A2), (A3$_\epsilon$) and (A4\( _p \)) for some $\epsilon>0$, \( p> 2 \). 
Then the vector-valued ASIP holds for mean zero $\BV$ observables $\phi:I\to\R^d$
(for any $\lambda>p/(4p-4)$).
\end{thm}

The ASIP implies a wide range of statistical properties such as the central
limit theorem and law of the iterated logarithm and their functional versions
(see for example~\cite{MN09}).    Many of these results require an $L^p$
condition, $p>2$, but it is natural to ask whether the central limit
theorem (CLT) and its functional version (FCLT) hold if the return time lies
in $L^2$.   We prove such a result 
under an additional assumption on the density of the invariant measure $\hat\mu$ for the induced Rychlik map $F$.
Let $m$ denote Lebesgue measure and write $d\hat\mu=h\,dm$.
(Since the vector-valued (F)CLT is no
more difficult than the scalar case, we suppose for simplicity that $d=1$.)

\begin{thm} \label{thm-CLT}
Suppose that \( f: I \to I  \) satisfies assumptions  (A1), (A2), (A3$_1$)
and (A4\( _2 \)).
Suppose moreover that $h^{-1}\in L^1(m)$.
Then the CLT holds for mean zero $\BV$ observables $\phi:I\to\R$.
That is, there exists $\sigma^2\ge0$ such that
\[
\lim_{n=\infty}\mu(n^{-\frac12}\phi_n<c)=P(G<c), \enspace\text{for all $c\in\R$},
\]
where $G$ is a normal random variable with mean zero and variance $\sigma^2$.

Moreover, the FCLT holds: define $W_n(t)=n^{-\frac12}\phi_{nt}$ for $t=j/n$,
$j=0,1,2,\dots,n$ and linearly interpolate to form a random element $W_n\in
C([0,1])$ (the space of continuous functions on $[0,1]$ with the sup-norm).   
Let $W(t)$ be a Brownian motion
with mean zero and variance parameter $\sigma^2$.   Then $W_n$ converges weakly to $W$ in $C([0,1])$.
\end{thm}

It is immediate from Rychlik~\cite{Rychlik83} that the density $h$ lies in
$L^\infty$ and even in $\BV$.   Unfortunately,
it is not known how to specify a condition directly on the map \( f \) to ensure that 
 $h^{-1}\in L^1(m)$ in the general setting of this paper, though it may be possible to verify this assumption in specific situations.   In the Markov context, there are classical assumptions on the branches of the induced map such as ``finite range structure'' which are used to ensure this.  However, the flexibility of the approach of Ara\'ujo {\em et al.}~\cite{ALV09} and in the current paper 
is achieved by avoiding such assumptions.
Remarkably, the only result that suffers as a consequence is
Theorem~\ref{thm-CLT}.

\subsection{Statement of results: Decay of correlations}
\label{subsec:decay}

Our second main set of results concerns decay of correlations.  As a consequence
we prove a result about large deviations.

In general $f:I\to I$
need not even be mixing.  However, there is a 
spectral decomposition into basic sets that are ergodic and
moreover mixing up to a finite cycle.     We restrict to such a basic set
(which we relabel as $I$) and assume that it is mixing.   
Given observables $v,w:I\to\R$, we define the correlation function
$\rho_{v,w}(n)=\int_I v\,w\circ f^n\,d\mu-\int_I v\,d\mu\int_I w\,d\mu$.

First we state a result about exponential decay of correlations.

\begin{thm} \label{thm-expdecay}
Suppose that \( f: I \to I  \) is mixing and satisfies assumptions  (A1), 
(A2), and that there exist constants $c_0,C_0\ge1$ such that 
$d_nE_n\ge C_0e^{c_0n}$ for all $n\ge1$.     
Then there exists $c>0$, $C>0$ such that
\[
|\rho_{v,w}(n)|\le C\|v\| |w|_\infty e^{-cn},
\]
for all $v\in\BV$, $w\in L^\infty$, $n\ge1$.
\end{thm}

\begin{rmk} Note that Theorem~\ref{thm-expdecay} covers completely the situation
discussed before the statement of Theorem~\ref{thm-ALV},
where the derivative grows exponentially fast, and the recurrence is subexponential or exponential at a sufficiently slow rate.

Theorem~\ref{thm-expdecay} seems to subsume all previous
results in the literature about
exponential decay of correlations for interval maps with critical points and singularities.   
In particular, it extends the afore-mentioned results of D\'iaz-Ordaz {\em et al.}~\cite{DiazHollandLuzzatto06} and also results of Young~\cite{Young92} who used related methods (inducing to a map satisfying Rychlik's conditions) for
certain quadratic maps.
\end{rmk}

Next, we consider polynomial decay of correlations.

\begin{thm} \label{thm-decay}
Suppose that \( f: I \to I  \)  is mixing and
satisfies assumptions  (A1), (A2) and (A3$_p$) for some $p>1$.  
Then for any $q>0$, there exists $\delta>0$, $C>0$ such that
\[
|\rho_{v,w}(n)|\le C\|v\| |w|_\infty\Bigl\{\sum_{j>\delta n}\mu(\tau>j)+n\mu(\tau>\delta n)+n^{-q}\Bigr\},
\]
for all $v\in\BV$, $w\in L^\infty$, $n\ge1$.
\end{thm}

\begin{rmk}
(1) 
Under the hypotheses of Theorem~\ref{thm-decay}, if $\mu(\tau>n)=\mathcal O(n^{-(\beta+1)})$ for some $\beta>0$, then taking $q=\beta$ we
obtain the optimal decay rate $\rho_{v,w}(n)=\mathcal O(1/n^{\beta})$.
A surprising aspect, inherited from~\cite{MTapp}, is that we obtain optimal
results with minimal assumptions on growth rates of derivatives and recurrence
rates (since any $p>1$ suffices regardless of the size of $\beta$).

\noindent (2)
In particular, if the hypotheses of Theorem~\ref{thm-decay} hold,
and $\mu(\tau>n)$ decays superpolynomially (faster than any
polynomial rate), then we obtain superpolynomial decay of correlations.
Hence, if $f$ is mixing and satisfies (A1), (A2), (A3$_p$) for some $p>1$, and (A4$_p$) for
all $p$, then we obtain superpolynomial decay of
correlations (since (A4$_p$) implies the required assumption on $\mu(\tau>n)$).

\noindent (3)
In the absence of singularities, i.e.\ for smooth maps, stronger versions of these  results are known.  We
mention the remarkable result by Rivera-Letelier \& Shen~\cite{RiveraShen} who prove superpolynomial decay of correlations for multimodal maps
under the assumption that $D_n\to\infty$ for all critical points
(no growth rate required).   An interesting open problem is whether  this
result can be generalized to interval maps with critical points and
singularities.
\end{rmk}

\begin{cor} \label{cor-LD}
Suppose that \( f: I \to I  \)  is mixing and
satisfies assumptions  (A1), (A2) and (A3$_p$) for some $p>1$.  
Suppose further that $\mu(\tau>n)=O(n^{-(\beta+1)})$ for some $\beta>0$.
Then we obtain the following (optimal) large deviation estimate:
For any $v\in\BV$ with $\int_Xv\,d\mu=0$ and any $\epsilon>0$, there
exists a constant $C>0$ such that
$\mu(|\sum_{j=0}^{n-1}v\circ f^j|\ge \epsilon n)\le Cn^{-\beta}$ for all $n\ge1$.
\end{cor}

\begin{proof}   This follows immediately from Theorem~\ref{thm-decay}
by~\cite[Theorem~1.2]{M09} which slightly improves a result of~\cite{MN08}.
\end{proof}

\subsection{Strategy of the proofs}
\label{subsec:strategy}

We end this introduction with a few remarks about the strategy of the proofs.
Given $\phi:[0,1]\to\R^d$, we
define the {\em induced observable} $\Phi:[0,1]\to\R^d$ by setting
$\Phi(y)=\sum_{j=0}^{\tau(y)-1}\phi(f^jy)$.
If $\int_I\phi\,d\mu=0$, then $\int_I\Phi\,d\hat\mu=0$.
The main part of the proof of Theorems~\ref{thm-ASIP} and~\ref{thm-CLT} is
to establish the corresponding statistical limit laws for the induced
observable $\Phi$ under the induced dynamical system $F$.   The
limit laws for $\phi$ follow in a by now standard way~\cite{Gouezel07,MT04,MZprep}.  

At the level of the induced map, it would follow easily from~\cite{Rychlik83}
that the (F)CLT holds for observables in $\BV$.    Extra work is required
since the induced observable $\Phi$ is only piecewise $\BV$ (on elements of the partition $\alpha$ with norm that grows with $\tau$).
Our method closely follows the approach of Melbourne \& Nicol~\cite{MN05} for 
piecewise H\"older observables and Gibbs-Markov maps.    A scalar
ASIP could also be proved along the same lines, and a vector-valued ASIP following~\cite{MN09}, but these would again require a condition on $h^{-1}$.   
Instead we apply a recent result of Gou\"ezel~\cite{Gouezel10} which sidesteps
this issue.

For decay of correlations, the strategy is somewhat different.  
The induced map $F$ still plays a major role but 
we do not consider induced observables.
Instead we use the method of operator renewal sequences~\cite{Sarig02,Gouezel04} to relate decay rates of the transfer operator for $f$ to
information about the transfer operator for $F$.  However certain key
estimates of first return operators in~\cite{Sarig02,Gouezel04} are
unavailable since we are not assuming long branches for $F$.
A modification of the method due to Melbourne \& Terhesiu~\cite{MTapp}
avoids sharp estimates for these operators and still yields optimal results.

In Section~\ref{sec-ALV}, we recast Theorems~\ref{thm-ASIP} and~\ref{thm-CLT}
in terms of conditions on the induced map $F$.
In Section \ref{sec:ind}, we prove limit theorems for the induced map. In Section \ref{sec-orig}, we show how the limit laws pass to the original map.
In Section~\ref{sec-decay}, we prove Theorems~\ref{thm-expdecay}
and~\ref{thm-decay}.

\section{Background from~\cite{ALV09}}
\label{sec-ALV}

We recall some aspects of the construction from~\cite{ALV09} that are needed for our results, in particular the definition of the partition \( \alpha \) and the associated return time function \( \tau \).
First $\delta$ is chosen small enough (prescribing the neighbourhood $\Delta$ of the critical/singular set $\mathcal{C}$).  
To each $x\in \Delta$ is associated an integer $b(x)\ge1$, called the {\em binding period} of \( x \),
with certain properties listed in Proposition~\ref{prop-bind} below.    Next, $q_0\ge1$ is fixed sufficiently large.
For $x\in I$, define $\ell_0(x)=\min\{j\ge0:f^jx\in\Delta\}$.
The return time function $\tau:I\to\R$ is then given by $\tau(x)=q_0$
if $\ell_0(x)\ge q_0$ and $\tau(x)=\ell_0(x)+b(f^{\ell_0(x)}(x))$ otherwise.
In the latter case, we say that $x$ has binding period $b(x)=b(f^{\ell_0(x)}(x))$.
The partition \( \alpha \) has the property that the binding period $b$ and return time function $\tau$ are constant on partition elements and that the induced map $F=f^\tau$ is smooth and monotone on partition elements.
Given $b\ge1$, we define $\alpha(b)$ to consist of those partition elements $a$ containing points with binding period $b$ and we let
 $\alpha(0)$ consist of the remaining partition elements.

\begin{prop}[Ara\'ujo {\em et al.}~\cite{ALV09}] \label{prop-bind}
There exist constants $C,M>0$ such that
\begin{itemize}
\item[(a)]  $\#\alpha(b)\le M$ for all $b\ge0$.
\item[(b)]  If $a\in\alpha(b)$, then
$\sup_a|1/F'|\le CE_b^{-1}$.
\item[(c)]  If $a\in\alpha(b)$, then
$\var_a|1/F'|\le C(1+\log d_{b-1}^{-1})d_{b-1}^{-1}E_{b-1}^{-1}$.
\end{itemize}
\end{prop}

\begin{proof}
Part (a) follows from the construction in~\cite{ALV09} since $f$
has only finitely many branches.
Parts (b) and (c) are~\cite[Equation~(19)]{ALV09} 
and~\cite[Equation~(20)]{ALV09} respectively.~
\end{proof}

We use Proposition~\ref{prop-bind} to convert conditions (A1)--(A4)
into conditions on the induced map $F$.
For $p\ge0$, introduce the conditions
\begin{description}
\item[(F1$_p$)]   $\sum_{a\in\alpha}{\SMALL\sup}_a(1/|F'|)\tau(a)^p<\infty$.
\item[(F2$_p$)]   $\sum_{a\in\alpha}{\SMALL\var}_a(1/|F'|)\tau(a)^p<\infty$.
\end{description}
Here \( \SMALL\sup_{a}(1/|F'|)=|1_a(1/F')|_\infty \), and \( \var\limits_a(1/F') \) denotes the variation of the function \( 1/F' \) on the interval \( a\in\alpha \).

\begin{cor} \label{cor-ALV}
\begin{itemize}
\item[(a)]  If $\sum_{n=1}^\infty
n^p E_n^{-1}< \infty$, then 
(F1$_p$) holds and
$\tau\in L^p(m)$.
\item[(b)]  If $\sum_{n=1}^\infty
n^pd_n^{-1}\log d_n^{-1}E_n^{-1}<\infty$,
then (F2$_p$) holds.
\end{itemize}
\end{cor}

\begin{proof}
It follows from the definition of $\tau$ and Proposition~\ref{prop-bind}(a,b)
that
\[
\sum_{a\in\alpha}{\SMALL\sup}_a(1/|F'|)\tau(a)^p
=\sum_{b=0}^\infty\sum_{a\in\alpha(b)}{\SMALL\sup}_a(1/|F'|)\tau(a)^p
\le CM \sum_{b=0}^\infty E_b^{-1}(b+q_0)^p<\infty.
\]
Also, by the mean value theorem $m(a)\le \sup_a(1/|F'|)m(Fa)\le \sup_a(1/|F'|)$,
so it follows similarly that
\[
\int|\tau|^p\,dm=\sum_{a\in\alpha}m(a)\tau(a)^p<\infty,
\]
proving part (a).
In the same manner, part (b) follows from Proposition~\ref{prop-bind}(a,c).
\end{proof}

Conditions (F1$_0$) and (F2\( _{0} \)) imply that \( F \) satisfies the \emph{Rychlik conditions} \cite{Rychlik83} and thus  admits an \( F \)-invariant absolutely continuous probability measure \( \hat\mu \) with a bounded density \( h= d\hat\mu/dm\in \BV \). The condition $\tau\in L^1(m)$ then implies, by standard arguments, that there exists a corresponding \( f \)-invariant absolutely continuous probability measure \( \mu \). 
  
The crucial ingredient in the proof of  Theorems~\ref{thm-ASIP}
and~\ref{thm-CLT} turns out to be existence of an induced map satisfying  
(F1$_1$) and (F2\( _{\epsilon} \)) and a condition on the integrability of the return time function \( \tau \). Assumptions (A1) and (A2) are important in the construction of such an induced map, but become redundant once such a map is assumed. By  Corollary~\ref{cor-ALV},  Conditions (A3\( _{p} \)) and (A4\(_{p}  \)) imply (F1$_1$) and (F2\( _{p} \)) and thus  we reduce Theorems~\ref{thm-ASIP}
and~\ref{thm-CLT} to the following.

\begin{thm}   \label{thm-ASIP2}
Suppose that \( f: I \to I \) admits an induced map satisfying (F1$_1$) and (F2\( _{\epsilon} \)) for some \( \epsilon>0 \) and that $\tau\in L^p$ for some $p>2$. 
Let $\phi:I\to\R^d$ be an observable in $\BV$
with $\int_{I}\phi\,d\mu=0$.   
Then $\phi$ satisfies the vector-valued ASIP for any  $\lambda>p/(4p-4)$.
\end{thm}

\begin{thm}   \label{thm-CLT2}
Suppose that \( f: I \to I \) admits an induced map satisfying (F1$_1$) and (F2\( _{\epsilon} \)) for some \( \epsilon>0 \) and that $\tau\in L^2$. 
Suppose further that $h^{-1}\in L^1(m)$. 
Let $\phi:I\to\R^d$ be an observable in $\BV$
with $\int_{I}\phi\,d\mu=0$.   
Then $\phi$ satisfies the CLT and functional CLT.
\end{thm}

\section{Limit theorems for the induced map}
 \label{sec:ind}

For generality and coherency with notation to be used below, we let $Y=I$ and let 
 $F:Y\to Y$ be a piecewise uniformly
expanding map with respect to a partition $\alpha$ of intervals and \( \mu_Y \) an ergodic, \( F \)-invariant, absolutely continuous probability measure with density \( h=d\mu_Y/dm \). 
Let $\tau:\alpha\to\Z^+$ be a ``weight function'' (which will of course be the return time function in the application of these results to be given below, but for the moment we are not assuming that \( F \) is an induced map) and define
\[
\|\Phi\|_{\BV_\tau}:=\sup_{a\in\alpha}\left\{\frac{{\SMALL\sup_a}(\Phi)+\var_a(\Phi)}{\tau(a)}\right\}.
\]
Then we define the space of functions \( \Phi: Y \to \R^d \) of \emph{ bounded weighted  variation} as 
\[ 
\BV_{\tau}:=\{\Phi: \|\Phi\|_{\BV_\tau}<\infty\}. 
 \] 
In particular, if $\tau\equiv1$, then  $\|\,\|_{\BV_\tau}=\|\,\|_{\BV}$
and $\BV_\tau=\BV$.

\begin{prop} \label{prop-ASIPinduce}
Suppose that (F1$_1$) and (F2\( _{\epsilon} \)) hold  for some \( \epsilon>0 \) . 
Let $p>2$.
Then the vector-valued ASIP holds for the induced map \( F:Y\to Y \) for mean zero observables 
$\Phi\in \BV_\tau\cap L^p(Y,\mu_Y)$.   
\end{prop}

\begin{prop}   \label{prop-CLTinduce}
Suppose (F1$_1$) and (F2\( _1 \)) hold and that $h\in\BV$,  $h^{-1}\in L^1(m)$. 
Then the CLT and functional CLT hold for the induced map $F:Y\to Y$
for mean zero observables $\Phi\in\BV_\tau\cap L^2(Y,\mu_Y)$.
\end{prop}

We prove these two Propositions in the following two subsections, then in the next section we show how they imply the Theorems above. 
We let $L$ and $P$ denote the transfer operators for $F$ with respect to
Lebesgue measure $m$ and the invariant measure $\mu_Y$ respectively: 
\[
\int_Y Lv\cdot w\,dm = \int_Y v\cdot w\circ F\,dm
\quad \text{ for all }v\in L^1(m), w\in L^\infty(m),\]
and
\[
\int_Y Pv \cdot w\,d\mu_Y = \int_Y v\cdot w\circ F\,d\mu_Y
\quad \text{ 
for all }
v\in L^1(\mu_Y), 
w\in L^\infty(\mu_Y).
\]
We have $Lh=h$, $P1=1$ and $P=h^{-1}Lh$.
Also, $(Lv)(x)= \sum_{Fy=x}v(y)/|F'(y)|$.

\subsection{ASIP for the induced map}
\label{sec-ASIPinduce}

In this subsection, we prove Proposition~\ref{prop-ASIPinduce}. We assume throughout assumptions (F1$_1$) and (F2\( _{\epsilon}) \) of the Proposition for some \( \epsilon>0 \). 
Our goal is to apply a general result of  Gou\"ezel~\cite[Theorem~2.1]{Gouezel10}.
Given $\Phi\in\BV_\tau$, $t\in\R^d$ and \( v\in \BV \), let
\[
L_tv:=L(e^{it\Phi}v)
\quad \text{ and } 
\quad 
P_tv:=P(e^{it\Phi}v).
\] 
Here $t\Phi$ is shorthand for $t\cdot\Phi$.
Note that $P_t=h^{-1}L_th$. We also define the Banach space $\mathcal{B}=h^{-1}\BV$ with norm $\|v\|_{\mathcal{B}}=\|hv\|_{\BV}$. 

\begin{lemma}  \label{lem-SG}
There exist constants $C>0$, $\gamma\in(0,1)$ such that
\begin{itemize}
\item[(a)] 
$\|L^nv-h\int v\,dm\|_{\BV}\le C\gamma^n\|v\|_{\BV}$
for all $n\ge1$, $v\in\BV$.
\item[(b)] $\|L_t-L_0\|_{\BV}\to0$ as $t\to0$.
\end{itemize}
\end{lemma}

\begin{proof}  Part (a) is in~\cite{Rychlik83} (using (F1$_0$) and (F2$_0$)).
Next, write
\begin{align*}
|(L_t-L_0)v(x)|&=|L(\{e^{it\Phi}-1\}v)(x)|\le \sum_{Fy=x}(1/|F'(y)|)|e^{it\Phi(y)}-1||v(y)| \\
&\le \sum_{a\in\alpha}{\SMALL\sup}_a|(1/F')(e^{it\Phi}-1)v|.
\end{align*}
Since $|e^{ix}-1|=\mathcal O(x^\epsilon)$ for any $\epsilon\in[0,1]$,
\begin{align*}
|(L_t-L_0)v|_\infty  & \le \sum_a {\SMALL\sup}_a(1/|F'|)|t|^\epsilon{\SMALL\sup}_a|\Phi|^\epsilon{\SMALL\sup}_a(v) \\ &
\le |t|^\epsilon (\|\Phi\|_{\BV_\tau})^\epsilon|v|_\infty \sum_{a\in\alpha}{\SMALL\sup}_a(1/|F'|)\tau(a)^\epsilon.
\end{align*}
Further,
\[
\var((L_t-L_0)v)\le 
\sum_{a\in\alpha}\var_a((1/F')(e^{it\Phi}-1)v)+2
\sum_{a\in\alpha}{\SMALL\sup}_a|(1/F')(e^{it\Phi}-1)v|.
\]
The second term is estimated as above.
Since $\var_a(e^{i\Phi}-1)\le \var_a(\Phi)$, 
\begin{align*}
& \var_a((1/F')(e^{it\Phi}-1)v) \le
\var_a(1/F'){\SMALL\sup}_a|(e^{it\Phi}-1)v| \\
&\qquad\qquad+
\var_a(v){\SMALL\sup}_a|(1/F')(e^{it\Phi}-1)|+
\var_a(e^{it\Phi}-1){\SMALL\sup}_a|(1/F')v| \\
& \qquad \le |t|^\epsilon\var_a(1/F'){\SMALL\sup}_a|\Phi|^\epsilon|v|_\infty+
|t|^\epsilon{\SMALL\sup}_a(1/|F'|){\SMALL\sup}_a|\Phi|^\epsilon\var(v) \\
&\qquad\qquad\qquad \qquad\qquad\qquad +
|t|\var_a(\Phi){\SMALL\sup}_a(1/|F'|)|v|_\infty \\
& \qquad \le |t|^\epsilon \|\Phi\|_{\BV_\tau}\|v\|_{\BV}\{
\var_a(1/F')\tau(a)^\epsilon + {\SMALL\sup}_a(1/|F'|)\tau(a)\}.
\end{align*}
By conditions (F1) and (F2\( _{\epsilon} \)), $\|L_t-L_0\|_{\BV}=\mathcal O(|t|^\epsilon)$ proving part (b).
\end{proof}

\begin{cor}  
There exist constants $C>0$, $\gamma\in(0,1)$ such that
\begin{itemize}
\item[(a)] 
$\|P^nv-\int v\,d\mu_Y\|_{\mathcal{B}}\le C\gamma^n\|v\|_{\mathcal{B}}$
for all $n\ge1$, $v\in\mathcal{B}$.
\item[(b)] $\|P_t-P_0\|_{\mathcal{B}}\to0$ as $t\to0$.
\end{itemize}
\end{cor}

\begin{proof}
From Lemma~\ref{lem-SG} we have 
\begin{align*}
\|P^nv-{\SMALL\int} v\,d\mu_Y\|_{\mathcal{B}}
& =\|h(P^nv-{\SMALL\int} hv\,dm)\|_{\BV}
=\|L^n(hv)-h{\SMALL\int} (hv)\,dm)\|_{\BV}
\\ & \le C\gamma^n\|hv\|_{\BV}=C\gamma^n\|v\|_{\mathcal{B}},
\end{align*}
and
\begin{align*}
\|(P_t-P_0)v\|_{\mathcal{B}} & 
= \|h(P_t-P_0)v\|_{\BV}
= \|(L_t-L_0)(hv)\|_{\BV}
\\ & \le  \|(L_t-L_0)\|_{\BV}\|hv\|_{\BV}
=  \|(L_t-L_0)\|_{\BV}\|v\|_{\mathcal{B}}
\end{align*}
so that $\|P_t-P_0\|_{\mathcal{B}}\le \|L_t-L_0\|_{\BV}\to0$ as $t\to0$.
\end{proof}

\begin{pfof}{Proposition~\ref{prop-ASIPinduce}}
We have verified the hypotheses (I) of
Gou\"ezel~\cite[Theorem~2.1]{Gouezel10}, so the result follows.
\end{pfof}

\begin{rmk}  From the proof it follows that condition (F1$_1$) can be replaced by the assumptions that (F1$_\epsilon$) holds for some $\epsilon>0$
and $\sum_{a\in\alpha}{\SMALL\sup}_a(1/|F'|)\tau(a)\var_a(\phi)<\infty$. 
\end{rmk}

\subsection{CLT for the induced map}
\label{sec-CLTinduce}

In this subsection, we prove Proposition~\ref{prop-CLTinduce} following
the approach in~\cite{MN05}.
We assume throughout the assumptions of the Proposition. 

\begin{lemma} \label{lem-tau}
The operator $L:\BV_\tau\to\BV$ is bounded.
\end{lemma}

\begin{proof}
Let $\Phi\in \BV_\tau$.  Then, using the definition of \( \|\cdot \|_{\BV_{\tau}} \), we have 
\[
|L\Phi(x)|\le \sum_{Fy=x}|\Phi(y)|/|F'(y)|
\le \|\Phi\|_{\BV_\tau} \sum_{a\in\alpha}{\SMALL\sup}_a(1/|F'|)\tau(a),
\]
so $|L\Phi|_\infty\le \{\sum_{a\in\alpha}{\SMALL\sup}_a(1/|F'|)\tau(a)\}\|\Phi\|_{\BV_\tau}$ and \( \sum_{a\in\alpha}{\SMALL\sup}_a(1/|F'|)\tau(a)< \infty \) by assumption (F1$_1$). 
Next, we note that
\[
\var(L\Phi)\le \sum_{a\in\alpha}\var_a(\Phi/F')+
2\sum_{a\in\alpha}{\SMALL\sup}_a|\Phi/F'|.
\]
The second sum is estimated as above using (F1$_1$).  For the first sum, we have
\begin{align*}
\var_a(\Phi/F') & \le \var_a(\Phi){\SMALL\sup}_a(1/|F'|)+ {\SMALL\sup}_a|\Phi|\var_a(1/F') 
\\ & \le 
\|\Phi\|_{\BV_\tau}\{{\SMALL\sup}_a(1/|F'|)+
\var_a(1/F') \}\tau(a).
\end{align*}
Hence $\var(L\Phi)\le \{
\sum_{a\in\alpha}\var_a(1/|F'|)\tau(a)+
4\sum_{a\in\alpha}{\SMALL\sup}_a(1/|F'|)\tau(a)
\}\|\Phi\|_{\BV_\tau}$ and \( \sum_{a\in\alpha}\var_a(1/|F'|)\tau(a)+
4\sum_{a\in\alpha}{\SMALL\sup}_a(1/|F'|)\tau(a)< \infty \) by (F1$_1$) and (F2\( _1 \)).
\end{proof}

\begin{cor}   \label{cor-tau} 
There exist $C>0$, $\gamma\in(0,1)$ such that $|P^n\Phi|_{L^2(Y,\mu_Y)}\le C\gamma^n|h^{-1}|_2\|h\|_{\BV}\|\Phi\|_{\BV_\tau}$ for all $n\ge1$. In particular $\chi:=\sum_{n=1}^\infty P^n\Phi \in L^2(Y,\mu_Y)$.
\end{cor}

\begin{proof}     
We use the fact that $h^{-1}\in L^1(m)$, or equivalently that  $h^{-1}\in L^2(Y,\mu_Y)$ (since 
\( d\mu_Y= h \,dm  \) we have   
\( \int (h^{-1})^2 d\mu = \int (h^{-1})^2 h \,dm  = \int h^{-1} \,dm \)).
   Thus,  writing $|\;|_2=|\;|_{L^2(Y,\mu_Y)}$ and 
 \(P^n\Phi  =  h^{-1} h P^n\Phi  \) we have 
\[
|P^n\Phi|_2\le |h^{-1}|_2|hP^n\Phi|_\infty\le |h^{-1}|_2\|hP^n\Phi\|_{\BV}
=|h^{-1}|_2\| L^{n-1}L(h\Phi)\|_{\BV}.
 \] 
Since $h\in\BV$, it follows 
that $h\Phi\in \BV_\tau$ and, by Lemma~\ref{lem-tau}, that $L(h\Phi)\in \BV$. 
Hence by Lemma~\ref{lem-SG}(a),
\[ 
|P^n\Phi|_2 \le C\gamma^{n-1}|h^{-1}|_2 \|L(h\Phi)\|_{\BV}.
 \] 
Finally, by Lemma~\ref{lem-tau},
$\|L(h\Phi)\|_{\BV}\le C \|h\Phi\|_{\BV_\tau}\le C\|h\|_{\BV}\|\Phi\|_{\BV_\tau}$
which gives the result. 
\end{proof}

\begin{pfof}{Proposition~\ref{prop-CLTinduce}}
The proof follows from Corollary~\ref{cor-tau} by 
a standard Gordin-type argument~\cite{Gordin69}. We sketch the steps of the argument.

Write $\Phi=\hat\Phi + \chi\circ F-\chi$.   Then 
it suffices to prove the CLT for 
$\hat\Phi$.   Now $\hat\Phi\in L^2(Y,\mu_Y)$ (since $\Phi,\chi\in L^2(Y,\mu_Y)$)
and it follows from the 
definitions that $\hat\Phi\in\ker P$.
This means that 
$\{\hat\Phi\circ F^n\}$ defines a sequence of reverse martingale increments.
Passing to the natural extension and applying the (F)CLT for ergodic $L^2$ 
martingales we obtain the required limit laws as $n\to-\infty$.
Since the limit laws are distributional, this is equivalent to the result as 
$n\to\infty$.
\end{pfof}

\section{Limit theorems for the original map}
\label{sec-orig}

We are now ready to complete the proof of Theorems~\ref{thm-ASIP}
and~\ref{thm-CLT}.
Starting with $F:Y\to Y$  and $\tau:Y\to{\mathbb N}$, 
we build a tower map $\tilde f:X\to X$ for which $F:Y\to Y$ is 
a {\em first} return map.
Define 
\[
X=\{(y,\ell)\in Y\times\Z: 0\le \ell\le \tau(y)-1\}
\]
and 
 \[
 \tilde f(y,\ell)=\begin{cases} (y,\ell+1), & \ell\le \tau(y)-2 \\
(Fy,0), & \ell=\tau(y)-1 \end{cases}.
\]
    Then $Y$ is identified with the base
$\{(y,0):y\in Y\}$ of the tower and $F=\tilde f^\tau$ is the first return map to $Y$.
Let $\mu_Y=\hat\mu$ denote the absolutely continuous $F$-invariant probability measure on $Y$.
Since $\tau$ is integrable, we can define a $\tilde f$-invariant
probability  measure $\mu_X=\mu_Y\times\nu/\int\tau\,d\mu_Y$ on $X$
where $\nu$ denotes counting measure.

Also, we have a natural projection $\pi:X\to I$ given by
$\pi(y,\ell)=f^\ell y$.   
This is a semiconjugacy: $\pi\circ \tilde f=f\circ\pi$.
The pushforward measure $\mu=\pi_*\mu$ is the absolutely
continuous $f$-invariant probability measure on $I$ constructed in~\cite{ALV09}.

Now suppose that $\phi:X\to\R^d$ is a mean zero $\BV$ observable as in
Theorem~\ref{thm-ASIP2} or Theorem~\ref{thm-CLT2}.  This lifts to 
an observable $\tilde\phi=\phi\circ\pi:X\to\R^d$ on the tower.
We can then define the
induced observable $\Phi:Y\to\R^d$ given by
\[
\Phi(y)=\sum_{\ell=0}^{\tau(y)-1}\tilde\phi(y,\ell)
=\sum_{\ell=0}^{\tau(y)-1}\phi(f^\ell y).
\]

The following abstract result shows that limit theorems
for the induced map are inherited by the original map.

\begin{prop} \label{prop-lift}
Suppose that $f:I\to I$ is ergodic and that
$\phi\in L^\infty(I)$ with $\int_I\phi\,d\mu=0$. 
\begin{itemize}
\item[(a)]
If $\tau\in L^p(Y,\mu_Y)$ for some $p>2$ and  $\Phi$ and $\tau-\int_Y\tau\,d\mu_Y$ satisfy the ASIP for $(Y,\mu_Y,F)$, then
$\phi$ satisfies the ASIP for $(I,\mu,f)$.
\item[(b)]
If 
$\tau\in L^2(Y,\mu_Y)$ and $\Phi$ and $\tau-\int_Y\tau\,d\mu_Y$ satisfy the (F)CLT for $(Y,\mu_Y,F)$, then
$\phi$ satisfies the (F)CLT for $(I,\mu,f)$.
\end{itemize}
\end{prop}

\begin{proof}   
The ASIP for $(Y,\mu_Y,F,\Phi)$ lifts by~\cite{DenkerPhilipp84,MT04} to
an ASIP for $(X,\mu_X,\tilde f,\tilde\phi)$.
The CLT lifts by~\cite[Theorem~A.1]{Gouezel07} (cf.~\cite[Theorem~1.1]{MT04}).  
The functional CLT lifts by~\cite{MZprep}.
It follows from the definition of $\pi$, $\tilde \phi$ and $\mu$ that the limit
laws for $(X,\mu_X,\tilde f,\tilde\phi)$ push down to the required limit laws 
for $(I,\mu,f,\phi)$.
\end{proof}

\begin{pfof}{Theorems~\ref{thm-ASIP2} and~\ref{thm-CLT2}}
The assumptions of Theorem~\ref{thm-ASIP2}
guarantee that $\tau\in L^p(Y,m)$ for some $p>2$.
Since $h\in\BV$, it follows that $\tau\in L^p(Y,\mu_Y)$.
Since $\phi\in\BV$, it follows from the definition of $\Phi$ that
$\Phi\in \BV_\tau\cap L^2(Y,\mu_Y)$.
By Proposition~\ref{prop-ASIPinduce}, $\Phi$ and $\tau$ satisfy the ASIP
for $(Y,\mu_Y,F)$.
Apply Proposition~\ref{prop-lift}(a) to obtain the ASIP for $\phi$.

The argument for Theorem~\ref{thm-CLT2} is identical.
\end{pfof}

\section{Decay of correlations}
\label{sec-decay}

In this section, we prove Theorems~\ref{thm-expdecay} and~\ref{thm-decay}.    By assumption $f:I\to I$
is mixing.   To begin with we make the simplifying assumption
that the tower map $\tilde f:X\to X$ constructed in Section~\ref{sec-orig} is mixing.  
This simplifying assumption is relaxed at the end of the section.

Let $\mathcal{B}_1(Y)$ denote the space of uniformly piecewise bounded
variation observables, namely those $v:Y\to\R$ such that
$\|v\|_{\mathcal{B}_1(Y)}=\sup_{a\in\alpha}\|1_a v\|_{\BV}<\infty$.
Take $\mathcal{B}(Y)=h^{-1}\mathcal{B}_1(Y)$
(with norm $\|v\|_{\mathcal{B}(Y)}=\|hv\|_{\mathcal{B}_1(Y)}$).  

\begin{prop} \label{prop-B}
\begin{itemize}
\item[(a)]
$\BV(Y)\subset\mathcal{B}_1(Y)\subset\mathcal{B}(Y)\subset L^1(Y,\mu_Y)$.
\item[(b)] The unit ball in $\mathcal{B}_1(Y)$ is compact in $L^1(Y,\mu_Y)$.
\end{itemize}
\end{prop}

\begin{proof}
The first inclusion in (a) is obvious, and the second inclusion holds since
$h\in \BV(Y)$.  Also $\mathcal{B}_1(Y)\subset L^\infty(Y)$ and
$h^{-1}\in L^1(Y,\mu_Y)$ (since $\int_Y h^{-1}\,d\mu_Y=\int_Y 1\,dm=1$)
so $\mathcal{B}(Y)\subset L^1(Y,\mu_Y)$ completing the proof of (a).

It is well-known that the unit ball in $\BV(Y)$ is compact in $L^1(Y,m)$, and hence
in $L^1(Y,\mu_Y)$ since the density $h$ is bounded.
Let $v_n\in\mathcal{B}_1(Y)$, $n\ge1$, with $\|v_n\|_{\mathcal{B}_1(Y)}\le1$.
Then for each $a\in\alpha$, $1_av_n$ lies in the unit ball in $\BV(Y)$.
Passing to a subsequence, we obtain a function $v:a\to\R$ with $\|1_av\|_{\BV(Y)}\le1$ such that $|1_a(v_n-v)|_1\to0$ for all $a$.   Putting these together,
we obtain a limit function $v\in\mathcal{B}_1(Y)$ with $\|v\|_{\mathcal{B}_1(Y)}\le1$ and a common subsequence $n_k$ such that
$\int_a|v_{n_k}-v|\,d\mu_Y\to0$ as $k\to\infty$ for all $a\in\alpha$.  
Since $|v_{n_k}-v|_\infty\le 2$, it follows that
$\int_Y|v_{n_k}-v|\,d\mu_Y\to0$ as $k\to\infty$.  
This concludes the proof of (b).
\end{proof}

Recall that $P$ denotes the transfer operator for the induced map
$F:Y\to Y$ with respect to the invariant measure $\mu_Y$.
Define operators $P_n:\mathcal{B}(Y)\to\mathcal{B}(Y)$ for $n\ge1$, and
$P(z):\mathcal{B}(Y)\to\mathcal{B}(Y)$ for $z\in\C,\,|z|\le1$:
\[
P_nv=P(1_{\{\tau=n\}}v), \qquad
P(z)=\sum_{n=1}^\infty P_nz^n.
\]

\begin{prop} \label{prop-H2}
\begin{itemize}
\item[(i)] The eigenvalue $1$ is simple and isolated in the spectrum of $P=P(1)$.
\item[(ii)]   For $z\neq1$, the spectrum of $P(z)$ does not contain $1$.
\end{itemize}
\end{prop}

\begin{proof}
Let $L$ denote the transfer operator for the induced map
$F:Y\to Y$ with respect to 
Lebesgue measure $m$.    Define $L_nv=L(1_{\{\tau=n\}}v)$ and
$L(z)=\sum_{n=1}^\infty L_nz^n$.
We show that properties (i) and (ii) are valid for 
$L(z):\mathcal{B}_1(Y)\to\mathcal{B}_1(Y)$.
Since $P(z)v=h^{-1}L(z)(hv)$ and $\mathcal{B}(Y)=h^{-1}\mathcal{B}_1(Y)$,
 the properties are inherited for $P(z)$.

First, let $z=1$.  Rychlik~\cite{Rychlik83} establishes the
``basic inequality'' $\|L^nv\|_{\BV(Y)}\le C(|v|_1+\gamma^n\|v\|_{\BV(Y)})$ for all
$v\in\BV(Y)$, $n\ge1$.   
Lemma~\ref{lem-tau} (with $\tau=1$) guarantees that $L:\mathcal{B}_1(Y)
\to\BV(Y)$ is a bounded operator, so that
$\|L^nv\|_{\mathcal{B}_1(Y)}\le C'(|v|_1+\gamma^n\|v\|_{\mathcal{B}_1(Y)})$ for all
$v\in\mathcal{B}_1(Y)$, $n\ge1$.   Combined with Proposition~\ref{prop-B}(b),
we deduce as in~\cite{Rychlik83} that 
$L:\mathcal{B}_1(Y)\to\mathcal{B}_1(Y)$ is quasicompact, and so has essential spectral radius
strictly less than $1$.
By ergodicity of $F$, $1$ is a simple eigenvalue, and property~(i) follows.

Aaronson~{\em et al.}~\cite[Proposition~4]{ADSZ04}, extend the argument
of~\cite{Rychlik83} to cover the case of $z\in\C$, $|z|=1$, and the case $|z|<1$ is simpler.
Hence for all $z\in\C$, $|z|\le1$, we have the basic inequality:
there exist constants $C>0$, $\gamma\in(0,1)$
such that $\|L(z)^nv\|_{\BV(Y)}\le C(|v|_1+\gamma^n\|v\|_{\BV})$ for all
$v\in\BV(Y)$, $n\ge1$.   Again, the operator $L(z):\mathcal{B}_1(Y)\to\BV(Y)$
is bounded (this is immediate from the proof of Lemma~\ref{lem-tau} since 
$L(z)v=L(z^\tau v)$ and $\tau$ is constant on partition elements
so $z^\tau$ does not contribute to the estimates).   Hence
$L(z)$ is quasicompact with essential spectral radius strictly less than $1$
for all $z$, $|z|\le1$.   

Next, we note that $L^n(z)v=L^n(z^{\tau_n}v)$ where 
$\tau_n=\sum_{j=0}^{n-1}\tau\circ F^j\ge n$.  By the basic inequality,
it is certainly the case that $|L^n|_\infty\le C$ for some constant $C$.
It follows that
$|L^n(z)v|_\infty=|L^n(z^{\tau_n}v)|_\infty\le C|z^{\tau_n}|_\infty|v|_\infty
\le C|z^n||v|_\infty$.  
In particular, if $\lambda$ is an eigenvalue then $|\lambda|\le |z|$.
Hence, condition (ii) is satisfied for all $z$ with $|z|<1$.

Finally, let $z=e^{i\theta}$, $0<\theta<1$, and
suppose that $L(e^{i\theta})v=v$ for some nonzero $v\in\mathcal{B}_1(Y)$.
The $L^2$ adjoint of $v\mapsto L(e^{i\theta})v=L(e^{i\theta\tau}v)$ is $v\mapsto e^{-i\theta}v\circ F$, and it follows easily that $e^{-i\theta} v\circ F=v$.
Form the tower map $\tilde f:X\to X$ as in Section~\ref{sec-orig}
and define $w:X\to\R$, $w(y,\ell)=e^{-i\ell\theta}v(y)$.
Then trivially $w\circ \tilde f=e^{-i\theta}w$ and the constraint
$e^{-i\theta} v\circ F=v$ means that $w(y,\tau(y))=w(Fy,0)$ so that 
$w$ is well-defined.  Certainly $w$ is  measurable.   
We have constructed
a nontrivial eigenfunction for $\tilde f$ which contradicts the assumption
that $\tilde f$ is mixing.  Hence, condition (ii) is satisfied for all
$z=e^{i\theta}$, $z\neq1$, completing the proof.
\end{proof}

\begin{lemma} \label{lem-*}
\begin{itemize}
\item[(a)] Suppose that \( f: I \to I  \) satisfies assumptions  (A1), (A2) and (A3$_p$) 
for some $p\ge1$.     Then 
$\sum_{n=1}^\infty n^{p-1} \sum_{j>n}\|P_j\|_{\mathcal{B}(Y)}<\infty$.
\item[(b)]   
Suppose that \( f: I \to I  \) satisfies assumptions  (A1) and  (A2),
and that there exist constants $c_0,C_0\ge1$ such that $d_nE_n\ge C_0e^{c_0n}$
for all $n\ge1$.     Then $\|P_n\|_{\mathcal{B}(Y)}=O(e^{-c_1n})$ for some $c_1>0$.
\end{itemize}
\end{lemma}

\begin{proof}
We prove part (a).   Part (b) is similar.

Define $L_nv=L(1_{\{\tau=n\}}v)$ regarded as an 
operator on $\mathcal{B}_1(Y)$ as in the proof of Proposition~\ref{prop-H2}.
We prove that 
$\sum_{n=1}^\infty n^{p-1} \sum_{j>n}\|L_j\|_{\mathcal{B}_1(Y)}<\infty$.
The required result is then immediate.

As in the proof of Lemma~\ref{lem-tau},
\begin{align*}
|L_nv|_\infty & \le \sum_{a\in\alpha:\tau(a)=n}{\SMALL\sup}_a(1/F')|v|_\infty, \\
\var(L_nv) & \le \sum_{a\in\alpha:\tau(a)=n}\var_a(v/F')+2
\sum_{a\in\alpha:\tau(a)=n}{\SMALL\sup}_a(v/F'),
\end{align*}
so it suffices to show that
\begin{align} \label{eq-MT}
 \sum_{n=1}^\infty n^{p-1}\sum_{a\in\alpha:\tau(a)>n}{\SMALL\sup}_a(1/F')<\infty, 
\quad
 \sum_{n=1}^\infty n^{p-1}\sum_{a\in\alpha:\tau(a)>n}\var_a(1/F')<\infty.
\end{align}
Recall that if $a\in\alpha(b)$, then $\tau(a)\in[b,q_0+b]$.
Hence if $\tau(a)>n$ then $a\in\alpha(b)$ with $b>n-q_0$, and so
\[
\sum_{a\in\alpha:\tau(a)>n}{\SMALL\sup}_a(1/F')
\le \sum_{b>n-q_0}\sum_{a\in\alpha(b)}{\SMALL\sup}_a(1/F')\le 
CM\sum_{b>n-q_0}E_b^{-1},
\]
by Proposition~\ref{prop-bind}(a,b).   
Note that (A3$_p$) implies (A4$_p$) which is equivalent to the condition
that $\sum_{n=1}^\infty n^{p-1}\sum_{b>n}E_b^{-1}<\infty$ verifying
the first condition in~\eqref{eq-MT}.
The second condition is verified in an identical manner using
Proposition~\ref{prop-bind}(a,c) and (A3$_p$).
\end{proof}

\begin{prop} \label{prop-exchange}
Let $v:I\to\R$ be $\BV$ and fix a function $j:Y\to\Z^+$ that is constant
on partition elements and satisfies $0\le j(y)<\tau(y)$.
Define $\hat v:Y\to\R$, $\hat v(y)=v(f^{j(y)}y)$.
Then $\|\hat v\|_{{\mathcal B}(Y)}\le \|h\|_{\BV}\|v\|_{\BV}$.
\end{prop}

\begin{proof}
Compute that
\begin{align*}
\|\hat v\|_{{\mathcal B}(Y)} & =
\|h \hat v\|_{{\mathcal B}_1(Y)}=
\sup_{a\in\alpha}\|1_ah \hat v\|_{\BV}
=\sup_{a\in\alpha}\|1_ah\, v\circ f^{j(a)}\|_{\BV}
\\ & \le \|h\|_{\BV}\sup_{a\in\alpha}\|1_a\, v\circ f^{j(a)}\|_{\BV}
= \|h\|_{\BV}\sup_{a\in\alpha}\|1_a v\|_{\BV}
\le \|h\|_{\BV}\|v\|_{\BV},
\end{align*}
as required.
\end{proof}

\begin{pfof}{Theorems~\ref{thm-expdecay} and~\ref{thm-decay}}
The Banach space $\mathcal{B}(Y)$ lies in $L^1(Y,\mu_Y)$ by Proposition~\ref{prop-B}(a)
and contains constant functions.   The conclusion of Lemma~\ref{lem-*}(a) with $p=1$ corresponds to~\cite[Hypothesis (\(\dagger\))]{MTapp}, and Proposition~\ref{prop-H2} corresponds to~\cite[Hypothesis~(H2)]{MTapp}.   Hence $F:Y\to Y$ is
a {\em good} inducing scheme in the sense of~\cite{MTapp}.
Moreover, Proposition~\ref{prop-exchange} means that
$\BV$ observables on $I$ are {\em exchangeable} in the sense of~\cite{MTapp}.

Lemma~\ref{lem-*}(a) with $p>1$ is the remaining hypothesis in~\cite[Theorem~1.11]{MTapp} yielding Theorem~\ref{thm-decay}.
Lemma~\ref{lem-*}(b) puts us in the position of~\cite[Example~5.3]{MTapp}
yielding Theorem~\ref{thm-expdecay}.
\end{pfof}

\begin{rmk}
Theorem~\ref{thm-decay} extends naturally to a much larger class of exchangeable
observables on $I$, namely those that are uniformly piecewise $\BV$.
More precisely, we define $\mathcal{B}_1(I)$ to consist of those $v:I\to\R$
such that $\|v\|_{\mathcal{B}_1(I)}=\sup_{a\in\alpha}\sup_{0\le j<\tau(a)}
\|1_{f^ja}v\|_{\BV}<\infty$.
To verify that $\mathcal{B}_1(I)$ is exchangeable in the sense of~\cite{MTapp}
it suffices to show that Proposition~\ref{prop-exchange} holds for $v\in\mathcal{B}_1(I)$, and the proof of this is the same as before.   
\end{rmk}

It remains to relax the assumption that the tower map $\tilde f:X\to X$ is mixing.
In general there exists an integer $k\ge1$ such that $\tilde f$ is mixing up
to a $k$-cycle. That is, $X$ is the disjoint union of $k$ subsets that are
cyclically permuted by $\tilde f$ and are mixing under $\tilde f^k$.
Label one of these subsets $X_k$.   Then we obtain a 
semiconjugacy $\pi_k:X_k\to I$ between $\tilde f^k:X_k\to X_k$ and $f^k:I\to I$.
By the proof in this section,
\[
|\rho_{v,w}(kn)|\le C\|v\| |w|_\infty\Bigl\{\sum_{j>\delta n}\mu(\tau>j)+n\mu(\tau>\delta n)+n^{-q}\Bigr\},
\]
for all $v\in\BV$, $w\in L^\infty$, $n\ge1$.
Redefining $C$ and $\delta$, and replacing $w$ by $w\circ f^j$, 
$0\le j\le k-1$, we obtain the required result.

\paragraph{Acknowledgements}
The research of IM was supported in part by EPSRC Grant EP/F031807/1.
IM is very grateful for the hospitality at ICTP, Trieste, where much of this research was carried out.

\end{document}